\documentclass[12pt]{amsart}

\usepackage{amsmath,amsfonts,amsthm}

\renewcommand{\theequation}{\thesection.\arabic{equation}}
\newtheorem{theorem}{Theorem}
\newtheorem{lemma}{Lemma}
\newtheorem{proposition}{Proposition}
\newtheorem{corollary}{Corollary}
\newtheorem{remark}{Remark}
\newtheorem{definition}{Definition}

\newcommand{\eqnsection}{
\renewcommand{\theequation}{\thesection.\arabic{equation}}
    \makeatletter
    \csname  @addtoreset\endcsname{equation}{section}
    \makeatother}
\eqnsection

\def\e{{\mathbb E}}
\def\p{{\mathbb P}}
\def\z{{\mathbb Z}}
\def\n{{\mathbb N}}

\def\ee{\mathrm{e}}
\def\d{\, \mathrm{d}}
\def\l{\ell}

\def\ds{\displaystyle}

\def\Z{{\mathbb Z}}

\def\N{{\mathbb N}}
\def\E{{\mathbb E}}

\def\w{\omega}

\def\indic{{\bf 1}}

\author[N. Enriquez]{Nathana\"el ENRIQUEZ}
\address{Laboratoire Modal'X, Universit\'e Paris 10, 200
Avenue de la R\'epublique, 92000 Nanterre, France}
\address{Laboratoire de Probabilit\'es et Mod\`eles Al\'eatoires, CNRS UMR 7599, Universit\'e Paris 6, 4
place Jussieu, 75252 Paris Cedex 05, France}
\email{nenriquez@u-paris10.fr}

\author[C. Sabot]{Christophe SABOT}
\address{Universit\'e de
Lyon, Universit\'e Lyon 1, Institut Camille Jordan, CNRS UMR 5208,
43, Boulevard du 11 novembre 1918, 69622 Villeurbanne Cedex,
France} \email{sabot@math.univ-lyon1.fr}

\author[O. Zindy]{Olivier ZINDY}
\address{Laboratoire Modal'X, Universit\'e Paris 10, 200
Avenue de la R\'epublique, 92000 Nanterre, France}
\address{Weierstrass Institute for Applied Analysis and Stochastics,
Mohrenstrasse 39, 10117 Berlin, Germany}
\email{olivier.zindy@u-paris10.fr}


\title[Aging and quenched localization for RWRE]{Aging and quenched localization for one-dimensional random walks in random environment in the sub-ballistic regime}

\begin{document}

\maketitle

\bigskip

{\footnotesize \noindent{\slshape\bfseries Abstract.} }
We consider transient one-dimensional random walks in a random environment with zero asymptotic speed. An aging phenomenon involving the generalized Arcsine law is proved using the localization of the walk at the foot of ``valleys" of height $\log t$. In the quenched setting, we also sharply estimate the distribution of the walk at time $t$.
\bigskip
\bigskip

\section{Introduction}

One-dimensional random walks in random environment have been the subject of constant interest in physics and mathematics for the last thirty years since they naturally appear in a great variety of situations in physics and biology.

In 1975, Solomon gave, in a seminal work
\cite{solomon}, a criterion of transience-recurrence for such
walks moving to the nearest neighbours, and shows that three different regimes can be distinguished:
the random walk may be recurrent,  or transient with a positive
asymptotic speed, but it may also be transient with zero asymptotic
speed. This  last regime, which does not exist among usual random
walks, is probably the one which is the less well understood and its
study is the purpose of the present paper.

Let us first recall the main existing results concerning the other
regimes. In his paper, Solomon  computes the asymptotic speed of
transient regimes. In 1982, Sinai states, in \cite{sinai}, a limit
theorem in the recurrent case. It turns out that the motion in this
case is unusually slow. Namely, the position of the walk at time $n$
has to be normalized  by $(\log n)^2$ in order to present a non
trivial limit. In 1986, the limiting law is characterized
independently by Kesten \cite{kesten86} and Golosov \cite{golosov2}.
Let us notice here that, beyond the interest of his result, Sinai
introduces a very powerful and intuitive tool in the study of
one-dimensional random walks in random environment. This tool is the
potential, which is a function on $\z$ canonically associated to the
random environment. The potential itself is a usual random walk when the
transition probabilities at each site are independent and
identically distributed (i.i.d.).

The proof by Sinai of an annealed limit law in the recurrent case is based on a quenched localization result. Namely, a notion of valley of the potential is introduced, as well as an order on the set of valleys. It is then proved that the walk is localized at time $t$, with a probability converging to 1, around the bottom of the smallest valley of depth bigger than $\log t$ surrounding the origin. An annealed convergence in law of this site normalized by $(\log t)^2$ implies the annealed limiting law for the walk.

  In the case of  transient random walks in random environment with zero asymptotic speed, the proof of the limiting law by Kesten, Kozlov and Spitzer \cite{kesten-kozlov-spitzer} does not follow this scheme. Therefore an analogous result to Sinai's localization in the quenched setting was missing. As we will see, the answer to this question is more complicated than in the recurrent case but still very explicit.

  In the setting of sub-ballistic transient random walks, the valleys we introduce are, like in \cite{enriquez-sabot-zindy-1} and \cite{peterson-zeitouni},  related to the excursions of the potential above its past minimum. Now, the  key observation is that with a probability converging to 1, the particle at time
 $t$ is located at the foot of a valley having depth and width of order $\log t$. Therefore, since the walk spends a random time of order $t$ inside a valley of depth $\log t$, it is not surprising that this random walk exhibits an aging phenomenon.

What is usually called aging is a dynamical out-of-equilibrium physical phenomenon observed in disordered systems like spin-glasses at low temperature, defined by the existence of a limit of a given two-time correlation function of the system as both times diverge keeping a fixed ratio between them; the limit should be a non-trivial function of the ratio. It has been extensively studied in the physics literature, see \cite{bouchaud-cugliandolo-kurchan-mezard} and therein references.

 More precisely, in our setting, Theorem \ref{t:main} expresses that, for each given ratio $h>1$, the probability that the particle remains confined within the same valley during the time interval $[t,th]$. This probability is expressed  in terms of the generalized Arcsine law, which confirms the status of universality ascribed to this law by Ben Arous and \v Cern\'y in their study of aging phenomena arising in trap models \cite{benarous-cerny06a}.

Recall that the  trap model is a model of random walk that was first  proposed by Bouchaud and Dean \cite{bouchaud, bouchaud-dean} as a toy model for studying this aging phenomenon.
In the mathematics litterature, much attention has recently been given to the trap model, and many aging result were derived from it, on $\z$ in \cite{fontes-isopi-newman} and \cite{benarous-cerny05},  on $\z^2$ in \cite{benarous-cerny-mountford}, on $\z^d$ $(d \ge 3)$ in \cite{benarous-cerny07b}, or on the hypercube in  \cite{benarous-bovier-gayrard03a,benarous-bovier-gayrard03b}. A comprehensive approach to obtaining aging results for the trap model in various settings was later developed in \cite{benarous-cerny07a}.

Let us finally mention that Theorem \ref{t:main} generalizes the aging result obtained by heuristical methods of renormalization by Le Doussal, Fisher and Monthus in \cite{LeDoussal-Fisher-Monthus}
in the limit  case when the bias  of the random walk defining the potential tends to 0 (the case when this bias is $0$ corresponding to the recurrent regime for the random walk in random environment).
The recurrent case, which also leads to aging phenomenon, was treated in the same article and rigorous arguments were later presented by Dembo, Guionnet and Zeitouni in \cite{dembo-guionnet-zeitouni}.

The second aspect of our work concerns localization properties of the walk and can be considered as the analog of Sinai's localization result  in the transient setting. Unlike the recurrent case, the random walk is not localized near the bottom of a single valley. Nevertheless, if one introduces a confidence threshold $\alpha$, one can say that, asymptotically, at time $t$,  with a probability converging to 1 on the environment, the walk is localized with probability bigger than
  $\alpha$  around the bottoms of a finite number of valleys having depth of order $\log t$. This number depends on $t$ and on the environment, but is not converging to infinity with $t$.
  Moreover, in Theorem \ref{t:main+} and Corollary \ref{c:main} we sharply estimate the probability for the walk of being at time $t$ in each of these valleys.

\section{Notation and main results}
Let $\omega:=(\omega_i, \, i \in \z)$ be a family of i.i.d. random
variables taking values in $(0,1)$ defined on $\Omega,$ which
stands for the random environment. Denote by $P$ the distribution
of $\omega$ and by $E$ the corresponding expectation. Conditioning
on $\omega$ (i.e. choosing an environment), we define the random
walk in random environment $X=(X_n, \, n \ge 0)$ on $\z^{\n}$ as a
nearest-neighbor random walk on $\z$ with transition probabilities
given by $\omega$: $(X_n, \, n \ge 0)$ is the Markov chain
satisfying $X_0=0$ and for $n \ge 0,$
\begin{eqnarray*}
P_\omega \left( X_{n+1} = x+1 \, | \, X_n =x\right) &=& \omega_x,
\\
P_\omega \left( X_{n+1} = x-1 \, | \, X_n =x\right)&=& 1-\w_x.
\end{eqnarray*}

\noindent We denote by $P_{\omega}$ the law of $(X_n, \, n \ge 0)$
and $E_{\omega}$ the corresponding expectation. We denote by $\p$
the joint law of $(\omega,(X_n)_{n \ge 0})$. We refer to Zeitouni
\cite{zeitouni} for an overview of results on random walks in
random environment.  Let us introduce
$$
\rho_i:= \frac{1-\omega_i}{\omega_i}, \qquad i \in \z.
$$
Our first main result is the following theorem which shows aging
phenomenon in the transient  sub-ballistic regime.

\medskip

\begin{theorem}
\label{t:main} Let $\omega:=(\omega_i, \, i \in \z)$ be a family
of independent and identically distributed random variables such
that
\begin{itemize}
  \item[{\it (a)}] there exists $0<\kappa<1$ for which
   $E \left[  \rho_0^{\kappa} \right]=1$ and $E \left[  \rho_0^{\kappa} \log^+ \rho_0 \right]<
   \infty,$
  \item [{\it (b)}] the  distribution  of $\log \rho_0$  is non-lattice.
\end{itemize}
Then, for all $h>1$ and all $\eta>0,$ we have
$$
\lim_{t\to\infty} \p( \vert X_{th}-X_t\vert \le \eta \log t )= {\sin(\kappa\pi)\over \pi} \int_{0}^{1/h} y^{\kappa-1} (1-y)^{-\kappa}\d y.
$$
\end{theorem}

\medskip

\begin{remark}\label{r:1}
The statement of Theorem \ref{t:main} could be improved in the
following way: the size of the localization window $\eta \log t$
could be replaced by any positive function $a(t)$ such that
$\lim_{t\to\infty} a(t)=+\infty$ and $a(t)=o(t^\kappa)$ (the
authors would like to thank Yueyun Hu who raised this question).
The extra constraint $a(t)=o(t^{\kappa})$ comes from the fact that
$t^\kappa$ is the order of the distance between successive valleys
where the RWRE can be localized. We did not write the proof of the
theorem in this more general version since it induces several
extra technicalities and makes the proof harder to read. Moreover
$\eta\log t $ represents an arbitrary portion of a typical valley
(which is of size of order $\log t$) where the RWRE can be
localized, and is therefore a natural localization window. 
\end{remark}

Let us now recall some basic result about $X_n$:
under the same assumptions (a)-(b), Kesten, Kozlov and Spitzer
\cite{kesten-kozlov-spitzer} proved that $X_n/n^{\kappa}$
converges in law to $C_\kappa({1\over
{\mathcal{S}_{\kappa}^{ca}}})^{\kappa}$ where $C_\kappa$ is a positive
parameter and $\mathcal{S}_{\kappa}^{ca}$ is the normalized
positive stable law of index $\kappa$, i.e. with Laplace transform
\begin{eqnarray*}
E[\ee^{-\lambda \mathcal{S}_{\kappa}^{ca}}]=\ee^{-
\lambda^{\kappa}}, \qquad \forall \lambda>0.
\end{eqnarray*}
In \cite{enriquez-sabot-zindy-2,enriquez-sabot-zindy-1} we gave a
different proof of this result and we were able to give an explicit
expression for the constant $C_\kappa.$

  The proof was based on a precise
analysis of the potential associated with the environment, as it was
defined by Sinai for its analysis of the recurrent case, see
\cite{sinai}. In this paper, we use the techniques developed in
\cite{enriquez-sabot-zindy-2,enriquez-sabot-zindy-1} to prove
Theorem \ref{t:main}.
The potential, denoted by $V= (V(x), \; x\in \z),$ is a function of
the environment $\omega.$ It is defined as follows:
$$
V(x) :=\left\{\begin{array}{lll} \sum_{i=1}^x \log \rho_i & {\rm if}
\ x \ge 1,
\\
 0 &  {\rm if} \ x=0,
\\
-\sum_{i=x+1}^0 \log \rho_i &{\rm if} \ x\le -1.
\end{array}
\right.
$$

\noindent Furthermore, we consider the weak descending ladder epochs
for the potential defined by $e_0:=0$ and
\begin{eqnarray*}
  e_i := \inf \{ k > e_{i-1}: \; V(k) \le V(e_{i-1})\}, \qquad i
  \ge 1,
\end{eqnarray*}

\noindent which play a crucial role in our proof. Observe that the sequence
$(e_i-e_{i-1})_{i \ge 1}$ is a family of i.i.d. random variables.
Moreover, classical results of fluctuation theory (see
\cite{feller}, p.~$396$), tell us that, under assumptions
$(a)$-$(b)$ of Theorem \ref{t:main},
\begin{eqnarray}
\label{excuinteg} E[e_1]< \infty.
\end{eqnarray}

\noindent Now, observe that the sequence $((e_i,e_{i+1}])_{i \ge 0}$ stands
for the set of excursions of the potential above its past minimum.
Let us introduce $H_i,$ the height of the excursion $[e_i,e_{i+1}]$
defined by 
\begin{eqnarray}
H_i := \max_{e_i \le k \le e_{i+1}}
\left(V(k)-V(e_i)\right), \qquad i \ge 0.
\end{eqnarray}
Note that the $(H_i)_{i
\ge 0}$'s are i.i.d. random variables.

For $t\in \N,$ we introduce the critical height
\begin{equation}\label{criticalheight}
h_t:=\log t- \log\log t.
\end{equation}
As in \cite{enriquez-sabot-zindy-1} we define the deep valleys from
the excursions which are higher than the critical height $h_t$. Let
$(\sigma(j))_{j \ge 1}$ be the successive indexes of excursions,
whose heights are greater than $h_t.$ More precisely,
\begin{eqnarray*}
\sigma(1)&:=&\inf \{ i \ge 0 : H_i \ge h_t\},
\\
\sigma(j)&:=& \inf \{ i >\sigma(j-1): H_i \ge h_t \}, \qquad j \ge
2.
\end{eqnarray*}

\noindent We consider now some random variables depending only on
the environment, which define the deep valleys.

\medskip

\begin{definition}
 \label{defvalley}
For all  $j\ge 1,$ let us introduce
\begin{eqnarray*}
b_j&:=&e_{\sigma(j)},
\\
a_j&:=&\sup \{ k \le b_j : \, V(k)-V(b_j) \ge D_t  \},
\\
T_{j}^{\uparrow}&:=&\inf \{ k \ge b_j : \, V(k)-V(b_j) \ge h_t \},
\\
\overline{d}_j&:=&e_{\sigma(j)+1},
\\
c_j&:=&\inf \{ k \ge b_j : \, V(k)=\max_{b_j \le x \le
\overline{d}_j} V(x) \},
\\
d_j&:=&\inf \{ k \ge \overline{d}_j: \, V(k)-V(\overline{d}_j) \le
-D_t \}.
\end{eqnarray*}

\noindent where $D_t:=(1+\kappa) \log t.$ We call
$(a_j,b_j,c_j,d_j)$ a deep valley and denote by $H^{(j)}$ the height
of the $j$-th deep valley.
\end{definition}

Moreover, let us introduce the first hitting time of $x,$ denoted by
$$
\tau(x):=\inf \{n \ge 1: \; X_n=x\}, \qquad x \in \z,
$$

and the index of the last visited deep valley at
time $t,$ defined by
$$
\l_t:=\sup \{n \ge 0: \; \tau(b_n) \le t\}.
$$
Before stating the quenched localization result, recall that $X$ is defined on the sample probability space $\z^{\n}.$ Then, let us introduce ${\bf e}=({\bf e}_i, \; i \ge 1)$ a sequence of i.i.d. exponential random variables with parameter $1,$ independent of $X.$ We define ${\bf e}$ on a probability space  $\Xi$ and denote its law by $P^{{\bf (e)}}.$
In order to express the independence between $X$ and   ${\bf e},$ we consider for each environment $\omega$, the probability space $(\z^{\n} \times \Xi, P_{\omega} \times P^{{\bf (e)}})$ on which we define $(X,{\bf e}).$

Furthermore, let us define the {\it weight} of the $k$-th deep valley by
$$  W_k(\omega):= 2 \sum_{{a_k\leq m\leq n}\atop{b_k\leq n\leq d_k}}e^{V_\omega(n)-V_\omega(m)}.$$
Moreover, let us introduce the following integer, for any $t\ge0,$
$$\l_{t,\omega}^{({\bf e})}:=\sup \Big\{i \ge 0: \;  \sum_{k=1}^{i}W_k(\omega){\bf e}_k \le t\Big\}.
$$
We are now able to state our second main result.
\medskip

\begin{theorem}
\label{t:main+} Under assumptions $(a)$-$(b)$ of Theorem
\ref{t:main}, we have,
\begin{itemize}
  \item[{\it (i)}] for all $\eta>0$,
  $$
   \lim_{t\to\infty} \p(\vert X_{t}-b_{\l_t}\vert \le \eta \log t)=1,
  $$
  \item [{\it (ii)}]  for all $\delta>0,$
$$
\lim_{t\to \infty} P\Big(d_{TV}(\l_{t},\l_{t,\omega}^{({\bf e})}+1)>\delta\Big)=0,$$
where $d_{TV}$ denotes the distance in total variation.
\end{itemize}
\end{theorem}
\medskip
\begin{remark}\label{r:2}
The statement of Theorem \ref{t:main+} could be improved in the following
way: the choice of the critical height $h_t$ is in some way
arbitrary and we could take for $h_t$ any positive function such
that $\lim_{t\to\infty} h_t=\infty$ and $e^{h_t}=o(t)$. The
meaning is that at time $t$ the RWRE is localized at the bottom of
a deep valley, deep meaning that its height $H$ is such that $e^H$
is of order $t$. Furthermore, as in Theorem \ref{t:main}, the size
of the localization window $\eta\log t$ could be replaced by any
positive function $a(t)$ such that $\lim_{t\to\infty}
a(t)=\infty$. 
\end{remark}



\medskip

We remark that we can easily deduce the following quenched
localization in probability result by assembling part $(i)$ and part $(ii)$ of
Theorem \ref{t:main+}. We precise that our quenched localization result is in probability
because one should not expect an almost sure result here, since no almost sure quenched limit results are expected to hold, see \cite{peterson-zeitouni}.
For $y<x,$ we denote by $E_{\omega}^{x}$ the expectation associated with the law $P_{\omega}^{x}$ of the particle in the environment $\omega,$ started at $x.$ 
\begin{corollary}
\label{c:main} Under assumptions $(a)$-$(b)$ of Theorem
\ref{t:main}, we have, for all $\delta,\eta>0$, that

$\ds P\Bigg(\sum_{i\geq 1}\Bigg\vert
P_{0,\omega}(\vert X_{t}-b_i\vert \le \eta \log t)-
P^{({\bf e})}\bigg(\sum_{k=1}^{i-1}W_k(\omega){\bf e}_k\leq t<
\sum_{k=1}^{i}W_k(\omega){\bf
e}_k\bigg)\Bigg\vert>\delta\Bigg)$
converges to 0, when $t$ tends to $\infty$.
\end{corollary}
\medskip

The content of this result is twofold. It first says that, with a probability converging to 1,  the process at time $t$ is concentrated near the bottom of a valley of depth of order $\log t$. It also determines, for each of these valleys,  the probability that, at time $t$, the particle lies at the bottom of it. This probability is driven by a renewal Poisson process which is skewed by the weights of each of these valleys.

This result may be of interest when trying  to get information on the environment  on the basis of the observation of a sample of trajectories of the particle. See \cite{cocco-monasson} for a recent example of this in a paper on DNA reconstruction.
\section{Notation}
\label{s:deepvalleys}

A result of Iglehart \cite{igle} which will be of constant
use, says that, under assumptions $(a)$-$(b)$ of Theorem
\ref{t:main}, the tail of the height $H_i$ of an excursion above
its past minimum is given by
\begin{equation}\label{iglehartthm}
     P (H_1 > h) \sim C_I \, \ee^{- \kappa h},
     \qquad h \to \infty,
\end{equation}
for a positive constant $C_I$ (we will not need its explicit
value).

The analysis done in
\cite{enriquez-sabot-zindy-2,enriquez-sabot-zindy-1} shows that on
the interval $[0,t]$, $t\in \N$,  the walk $X_n$ spends
asymptotically all its time trying to climb excursions of height of
order $\log t + C$ for a real $C$. Let us now introduce the integer
$$
n_t:=\lfloor t^\kappa \log\log t\rfloor .
$$
The integer $n_t$ will be use to bound the number of excursions
the walk can cross before time $t$. The strategy will be to show that we can neglect the time spent
between two excursions of size smaller than $h_t$, and to show
that at time $t$ the walk $X_t$ is close to the foot of an
excursion of height larger than $h_t$.

\subsection{The deep valleys}
\label{deepvalleys}

Let us define the number of deep valleys in the $n_t$ first excursions by
$$
K_t:=\sup\{j \ge 0 : \; \sigma(j)\le n_t\},
$$
which is the number of excursions higher than the critical height
$h_t$ in the $n_t$ first excursions.

\begin{remark}\label{r:ht}
This definition corresponds to the definition of deep valleys
introduced in \cite{enriquez-sabot-zindy-1} with $n=n_t$, but with a
different critical height. In \cite{enriquez-sabot-zindy-1} the
critical height was $h_n={1-\varepsilon\over \kappa} \log n,$ for
$\varepsilon$ such that $0<\varepsilon<1$. Here, we see that $h_{n_t}$
would be equal to $(1-\varepsilon)\log t+{1-\varepsilon\over \kappa}
\log\log\log t$ which is smaller than our critical height $h_t=\log
t -\log\log t$. This means that the deep valleys are higher and less
numerous in the present paper than in \cite{enriquez-sabot-zindy-1}. We will see that this choice makes possible the control of the localization of the particle in any neighborhood of size $\eta \log t$ around the bottom of the last visited valley (recall Part $(i)$ of Theorem \ref{t:main+}).
\end{remark}

\subsection{The $*$-valleys}
\label{*-valleys}

\noindent Let us first define the maximal variations of the potential
before site $x$ by:
\begin{eqnarray*}
 \label{incrempoten}
    V^{\uparrow}(x):= \max_{0 \le i \le j \le x} (V(j)-V(i)), \qquad x \in
    \n,
    \\
    V^{\downarrow}(x):= \min_{0 \le i \le j \le x} (V(j)-V(i)), \qquad x \in
    \n.
\end{eqnarray*}

\noindent  By extension, we introduce
\begin{eqnarray*}
 \label{incremupgene}
    V^{\uparrow}(x,y):= \max_{x \le i \le j \le y} (V(j)-V(i)),
    \qquad  x<y,
    \\
\label{incremdowngene}
    V^{\downarrow}(x,y):= \min_{x \le i \le j \le y} (V(j)-V(i)),
    \qquad x<y.
\end{eqnarray*}

The deep valleys defined above are not necessarily made of disjoint
portions of the environment. To overcome this difficulty we defined
another type of valleys, called $*$-valleys, which form a
subsequence of the previous valleys. By construction, the $*$-valleys are made
of disjoint portions of environment and will coincide with
high probability with the previous valleys on the portion of the
environment visited by the walk before time $t$.
\begin{eqnarray*}
\gamma^*_1&:=&\inf \{ k \ge 0 : \, V(k) \le -D_t  \},
\\
T_{1}^{*}&:=&\inf \{ k \ge \gamma^*_1 : \,
V^{\uparrow}(\gamma^*_1,k) \ge h_t \},
\\
b^*_1&:=&\sup \{ k \le T_{1}^{*} : \, V(k)=\min_{0 \le x \le
T_{1}^{*}} V(x) \},
\\
a^*_1&:=&\sup \{ k \le b^*_1 : \, V(k)-V(b^*_1) \ge D_t  \},
\\
\overline{d}^*_1&:=&\inf \{ k \ge T_{1}^{*} : \, V(k) \le
V(b_1^*)\},
\\
c^*_1&:=&\inf \{ k \ge b^*_1 : \, V(k)=\max_{b^*_1\le x \le
\overline{d}^*_1} V(x) \},
\\
d^*_1&:=&\inf \{ k \ge \overline{d}^*_1 : \,
V(k)-V(\overline{d}^*_1) \le -D_t  \}.
\end{eqnarray*}

\noindent Let us define the following sextuplets of points by
iteration
$$
(\gamma^*_j,a^*_j,b^*_j,T_{j}^{*},c^*_j,\overline{d}^*_j,d^*_j):=(\gamma^*_1,a^*_1,b^*_1,T_{1}^{*},c^*_1,\overline{d}^*_1,d^*_1)
\circ \theta_{d_{j-1}^*}, \qquad j \ge 2,
$$
where $\theta_{i}$ denotes the $i$-shift operator.

\begin{definition}
 \label{def*valley}
We call a $*$-valley any quadruplet $(a^*_j,b^*_j,c^*_j,d^*_j)$
for $j \ge 1.$ Moreover, we shall denote by $K_t^*$ the number of
such $*$-valleys before $e_{n_t},$ i.e. $ K_t^*:= \sup \{ j \ge 0
: T_{j}^{*} \le e_{n_t} \}.$
\end{definition}

\noindent The $*$-valleys will be made of independent and identically
distributed portions of potential (up to some translation).

\section{Preliminary estimates}
\label{s:preliminaries}
\subsection{Good environments}
We define in this subsection the {\it good environments} in the same manner as we did in
 \cite{enriquez-sabot-zindy-1} to give a complete characterization of the limit law. Since the critical height considered here is not the same (see Remark \ref{r:ht}), the following results are not taken from  \cite{enriquez-sabot-zindy-1} but proved with the same ideas, that we recall in this subsection. Let us introduce the following series of events, which will occur with high probability when $t$ tends to infinity.
\begin{eqnarray*}
 A_1(t)&:=&\left\{ e_{n_t}\le C'n_t \right\},
 \\
  A_2(t) &:=& \left\{K_t\le (\log t)^{{1+\kappa\over 2}} \right\},
     \\
  A_3(t) &:=& \cap_{j=0}^{K_t} \left\{ \sigma(j+1)-\sigma(j) \ge t^{\kappa/2}
    \right\},
     \\
  A_4(t) &:=& \cap_{j=1}^{K_t+1} \left\{ d_j- a_j
    \le C'' \log t \right\},
\end{eqnarray*}
 where $\sigma(0):=0$ (for convenience of notation) and
$C'$, $C''$  stand for positive constants (large enough) which will be specified
below. In words, $ A_1(t)$ bounds the total length of the first $n_t$ excursions.
The event $ A_2(t)$ gives a control on the number of deep valleys while $ A_3(t)$ ensures that they are well separated and  $ A_4(t)$ bounds finely the length of each of them.

\begin{lemma}\label{l:preliminaires}
Let $A(t):= A_1(t)\cap A_2(t)\cap A_3(t)\cap A_4(t)$, then
$$
\lim_{t\to\infty} P(A(t))=1.
$$
\end{lemma}
\begin{proof} The fact that $P(A_1(t)) \to 1$ is a consequence of the law of large numbers.
Concerning $A_2(t)$ and $A_3(t)$, we know that the number of excursions higher
than $h_t$ in the first $n_t$ excursions is a binomial random variable with
parameter $(n_t, q_t)$ where $q_t:=P(H_1\ge h_t),$ from which we can easily deduce that $P(A_2(t)\cap A_3(t)) \to 1.$ For example, since
(\ref{iglehartthm}) implies $q_t\sim C_I e^{-\kappa h_t},$ $t \to \infty,$ we have that
$E\left[K_t\right]=n_tq_t\sim C_I\log\log t (\log t)^\kappa$. Using the Markov
inequality we get that $P(A_2(t))$ tends to 1, when $t$ tends to infinity.

The proof for $A_4(t)$  requires a bit more explanations. Since $K_t\le (\log t)^{{1+\kappa\over 2}}$ with probability tending to one, we only  have to prove, for $j\ge1$ that
$P(d_j-a_j \ge C'' \log t) = o( (\log t)^{-{1+\kappa\over 2}}).$
 Furthermore, observe that we can write $d_j-a_j=(d_j-\overline d_j)+(\overline d_j-T_{j}^{\uparrow})+(T_{j}^{\uparrow}-b_j)+(b_j-a_j).$ Therefore, the proof boils down to showing that, for each term in the previous sum, the probability that it is larger than ${C''\over 4} \log t$ is a $o( (\log t)^{-{1+\kappa\over 2}}).$  Here, we only prove that
\begin{equation}
\label{eq:A1}
P(T_{j}^{\uparrow}-b_j \ge {C'' \over 4} \log t) = o( (\log t)^{-{1+\kappa\over 2}}), \qquad t \to \infty,
 \end{equation}
the arguments for the other terms being similar and the results more intuitive.

Let us first introduce $T_h:=\inf \{x \ge 0: \, V(x)\ge h\}$ for any $h>0.$ Then, recalling (\ref{iglehartthm}),
we can write
\begin{equation}
\label{eq:A2}
P(T_{j}^{\uparrow}-b_j \ge {C''\over 4} \log t) \le C \ee^{\kappa h_t} P( {C'' \over 4} \log t \le T_{h_t} < \infty).
 \end{equation}
Denoting by $I(\cdot)$ the convex rate function associated with the potential, we apply Chebychev's inequality in the same manner as is done in the proof of the upper bound in Cramer's theorem (see \cite{denhollander}) and obtain that the probability on the right-hand side in (\ref{eq:A2}) is bounded above by
\begin{eqnarray}
\label{eq:A3}
 \sum_{k \ge \frac{C''}{4} \log t}  P(V(k) \ge h_t)
 \le \sum_{k \ge \frac{C''}{4} \log t}  \ee^{-k \, I\left({h_t \over k}\right)}
 \le  \sum_{k \ge \frac{C''}{4} \log t}  \ee^{-k \, I(0)} \le C t^{-\frac{C''}{4} \,
   I(0)}.
\end{eqnarray}
Now, let us recall that $h_t\le \log t$ by definition. Morever, observe that the assumption $(a)$ implies that $E \left[  \rho_0^{\kappa} \right]=1,$ which yields $I(0)>0$. Then, assembling (\ref{eq:A2}) and (\ref{eq:A3}) yields (\ref{eq:A1}) by choosing $C''$ larger than $4 \kappa / I(0),$ which concludes the proof of Lemma \ref{l:preliminaires}.
\end{proof}

\noindent The following lemma tells us that the $*$-valleys, which are i.i.d., coincide
with the sequence of deep valleys with an overwhelming probability
when $t$ goes to infinity.

\medskip

\begin{lemma}
\label{model1+2} If $A^*(t):=\{  K_t=K_t^* \, ; \,
(a_j,b_j,c_j,d_j)= (a^*_j,b^*_j,c^*_j,d^*_j), \, 1 \le j \le K_t^*
\},$ then we have that the probability $P(A^*(t))$ converges to
$1,$ when $t$ goes to infinity.
\end{lemma}

\begin{proof} By definition, the $*$-valleys
constitute a subsequence of the deep valleys, and $A^*(t)$ occurs as
soon as the valleys $(a_j,b_j,c_j,d_j)$ are disjoint for $1\le j\le
K_t$. Hence, we see that $A_3(t)\cap A_4(t)\subset A^*(t)$. Then, Lemma \ref{model1+2} is a consequence of Lemma \ref{l:preliminaires}.
\end{proof}

\subsection{Directed traps}
Let us first recall that it is well-known (see for example \cite{zeitouni}, formula (2.1.4)) that for
$r<x<s$,
\begin{equation}
    P_{\omega}^{x} \left( \tau(r)< \tau(s) \right) =
    \sum_{j=x}^{s-1} \ee^{V(j)} \left( \, \sum_{j=r}^{s-1} \ee^{V(j)}
    \right) ^{\! \! -1}.
    \label{zeitouni}
\end{equation}
Moreover, we introduce here the inter-arrival times, defined, for any $x,y \in \z,$ by
$$
\tau(x,y):=\inf\{k \ge 0: \, X_{\tau(x)+k}=y\}.
$$
With the two following lemmas, we prove that the particle never backtracks before $a_j$ after reaching the bottom $b_j$ of the $j$-th valley, uniformly in $1\le j \le K_t,$ and that it visits each of them only once.
\medskip
\begin{lemma}
 \label{l:DT}
 Defining $DT(t) := A(t) \cap  \bigcap_{j=1}^{K_t}
    \left\{ \tau(d_j,b_{j+1})< \tau(d_j,\overline{d}_j) \right\},$ we
    have
 \begin{eqnarray*}
\lim_{t\to\infty} \p(DT(t))=1.
 \end{eqnarray*}
\end{lemma}
\begin{proof}
Recalling that $K_t\le (\log t)^{{1+\kappa\over 2}}$ with probability tending to one, we have to prove, for $j\ge1,$ that $E[ {\bf 1}_{A(t) \cap \{ j \le K_t\}}  P_{\omega}( \tau(d_j,b_{j+1})> \tau(d_j,\overline{d}_j))]  =o( (\log t)^{-{1+\kappa\over 2}}),$ when $t$ tends to infinity. Therefore, applying the strong Markov property at $\tau(d_j),$ we need to prove that
 \begin{eqnarray}
 \label{eq:DT1}
E[ {\bf 1}_{A(t) \cap \{ j \le K_t\}}  P_{\omega}^{d_j}( \tau(b_{j+1})>\tau(\overline{d}_j))]  =o( (\log t)^{-{1+\kappa\over 2}}), \qquad t \to \infty.
 \end{eqnarray}
 By (\ref{zeitouni}) we get that $P_{\omega}^{d_j} \left( \tau( b_{j+1})> \tau(\overline{d}_j) \right)$ is bounded by $(b_{j+1}-d_j)
\ee^{V(d_j)-V(\overline{d}_j)+h_t}.$ Observe first that $b_{j+1}-d_j \le e_{n_t} \le C' n_t$ on  $A(t) \cap \{ j \le K_t\}.$  Then, recalling that $V(d_j)-V(\overline{d}_j)\le-D_t$ by definition (where $D_t=(1+ \kappa) \log t$) together with $h_t \le \log t$ yields (\ref{eq:DT1}) and concludes the proof of Lemma \ref{l:DT}.
  \end{proof}

\begin{lemma}
 \label{l:DT*}
 Defining $DT^*(t):=\bigcap_{j=1}^{K_t^*}
   \left\{ \tau(b_j^*,d_j^*)< \tau(b_j^*,a^*_{j})
    \right\},$ we have
     \begin{eqnarray*}
\lim_{t\to\infty} \p(DT^*(t))=1.
 \end{eqnarray*}
 \end{lemma}
\begin{proof} We omit the details here since the arguments are very similar to the proof of Lemma \ref{l:DT}.
 \end{proof}
Finally, we need to know that the time spent between the deep
valleys is small. Let us first recall the following technical result proved in Lemma $7$ of
\cite{enriquez-sabot-zindy-1}.
\begin{lemma}
 \label{l:expoH}
Let  $T^{\uparrow}$ be defined by $T^{\uparrow}(h):=\inf\{x \ge 0: \, V^{\uparrow}(x) \ge h \},$ for any $h\ge 0.$ Then, there exists $C>0$ such that,  for all $h,$
 \begin{eqnarray*}
\e_{\vert 0}\left[\tau(T^{\uparrow}(h)-1)\right] \le  C \ee^h,
 \end{eqnarray*}
where $\e_{\vert 0}$ denotes the expectation under the annealed law $\p_{\vert 0}$ associated with the random walk in random environment  reflected at $0.$
\end{lemma}
Now, we can prove that the time spent by the particle between the $K_t$ first deep valleys is negligible with respect to $t$ with an overwhelming probability when $t$ goes to infinity, which is the statement of the following lemma.
\begin{lemma}
 \label{l:IA}
 Let us introduce the following event
$$IA(t) :=A(t) \cap \left\{  \tau(b_1)+ \sum_{j=1}^{K_t}
\tau(d_j,b_{j+1})<{t\over \log\log t}
  \right\}.
$$
Then, we have
 \begin{eqnarray*}
\lim_{t\to\infty} \p(IA(t))=1.
 \end{eqnarray*}
\end{lemma}
\begin{proof} Recalling Lemma \ref{l:preliminaires}, Lemma \ref{l:DT} and using the Markov inequality, we only need to prove that $\E[{\bf 1}_{A(t) \cap DT(t)} (\tau(b_1)+ \sum_{j=1}^{K_t} \tau(d_j,b_{j+1}))]$ is  $o({t \over \log \log t}),$ when $t$ goes to infinity. For $y<x,$ let us denote by $E_{\omega,\vert y}^{x}$ the expectation associated with the law $P_{\omega,\vert y}^{x}$ of the particle in the environment $\omega,$ started at $x$ and reflected at site $y.$   Then, applying the strong Markov property at times $\tau(d_{K_t}),\dots,\tau(d_1),$ we get that the above expectation is smaller than
  \begin{eqnarray} \label{eq:IA1}
  E[{\bf 1}_{A(t) \cap DT(t)} \tau(b_1)]+  E\Big[{\bf 1}_{A(t) \cap DT(t)} \sum_{j=1}^{K_t} E_{\omega,\vert \overline d_j}^{d_j}[ \tau(b_{j+1})]\Big],
 \end{eqnarray}
 since $(X_{\tau(d_j)+n})_{n \ge 0}$ under $P_{\omega}$ has the same law as  $(X_{n})_{n \ge 0}$ under $P_{\omega,\vert \overline d_j}^{d_j}$ on $A(t) \cap DT(t).$ Concerning the second term of (\ref{eq:IA1}), we apply the strong Markov property for the potential at times $ \overline d_{K_t},\dots, \overline d_1,$ such that we get
  \begin{eqnarray*}
 E\Big[{\bf 1}_{A(t) \cap DT(t)} \sum_{j=1}^{K_t} E_{\omega,\vert \overline d_j}^{d_j}[ \tau(b_{j+1})]\Big] &\le& (\log t)^{{1+\kappa\over 2}}  \e_{\vert 0}\left[\tau(T^{\uparrow}(h_t)-1)\right]
 \\ &\le&  C (\log t)^{{1+\kappa\over 2}}e^{h_t} \le C
t (\log t)^{-{1-\kappa\over 2}},
 \end{eqnarray*}
 the second inequality being a consequence of Lemma \ref{l:expoH}. Now, let us mention that the bound $C e^{h_t}$ can be obtained in a similar way for the first term of (\ref{eq:IA1}), which yields that the expression in (\ref{eq:IA1}) is a $o({t \over \log \log t}),$ when $t$ tends to infinity and concludes the proof of Lemma \ref{l:IA}.
\end{proof}
\subsection{Localization in deep traps}
In a first step, we state a technical result which ensures that the potential does not have excessive fluctuations in a typical box and which will be very useful to control the localization of the particle in a valley.
\begin{lemma}
 \label{l:potfluctu}
If $F_\gamma(t):= \cap_{j=1}^{K_t} \left\{ \max\{V^{\uparrow}(a_j,b_j)\, ; \, -V^{\downarrow}
  (b_j,c_j) \, ; \, V^{\uparrow}(c_j,d_j)\} \le \gamma \log t  \right\},$ then we have, for any $\gamma>0,$
   \begin{eqnarray*}
\lim_{t\to\infty} \p(F_\gamma(t))=1.
 \end{eqnarray*}
\end{lemma}
\begin{proof} Observe first that  Lemma $14$ in
\cite{enriquez-sabot-zindy-1} implies that, for all $\varepsilon>0,$ the valleys with height larger that $(1-\varepsilon)\log t+{1-\varepsilon\over \kappa} \log\log\log t$ have fluctuations bounded by $\gamma \log t,$ with a probability tending to one, for any $\gamma>\varepsilon/\kappa.$ Now, since $h_t$ is larger than $(1-\varepsilon)\log t+{1-\varepsilon\over \kappa} \log\log\log t$ for any $\varepsilon>0$ (see Remark \ref{r:ht}), the deep valleys considered here are included in the valleys treated by Lemma $14$ in
\cite{enriquez-sabot-zindy-1} for any $\varepsilon>0,$ which concludes the proof of Lemma \ref{l:potfluctu}.
\end{proof}
For each deep valley, let us introduce the position $\overline c_i$ defined by
$$
\overline c_i :=\inf\{n\ge c_i : \; V(n)\le V(c_i)-h_t/3\}.
$$
We first need to know that during its sojourn time inside a deep
valley, the random walk spends almost all its time inside the
interval $(a_i,c_i)$. This is a consequence of the following lemma.
\begin{lemma}\label{l:LT}
Let $LT(t)$ be the event
$$
LT(t):=\bigcap_{i=1}^{K_t}\left\{\tau(\overline c_i, d_i)\le {t\over
\log t}\right\}.
$$
Then,
   \begin{eqnarray*}
\lim_{t\to\infty} \p(LT(t))=1.
 \end{eqnarray*}
\end{lemma}
\noindent This result just means that at the time scale $t$, if the
walk reaches $\overline c_i$, then soon after it exits the deep valley
$(a_i,d_i)$.
\begin{proof}
Recalling Lemma \ref{l:preliminaires} and Lemma \ref{l:potfluctu}, we only have to prove that
$$
\p\left(\tau(\overline{c}_j,d_j) > {t\over
\log t} \, ; \,  A_4(t)  \, ; \, F_\gamma(t) \, ; \,   j \le K_t\right)=o((\log t)^{-{1+\kappa\over 2}}), \qquad t \to \infty,
$$
for any $j \ge 1.$
Now, applying the strong Markov property at $\tau(\overline{c}_j)$, we get that the previous probability is bounded by
$$
 E\left[{\bf 1}_{A_4(t)  \cap F_\gamma(t) \cap  \{ j \le K_t\}} \left(P_{\omega,\vert c_j}^{\overline{c}_j}\left(\tau(d_j) > {t /
\log t} \right) +P_{\omega}^{\overline{c}_j}\left(\tau(c_j)<\tau(d_j)   \right)\right)  \right].$$
Concerning the first term, we use the fact that $E_{\omega, \vert c_j}^{\overline{c}_j}[\tau(d_j)] \le \sum_{c_j \le u \le v \le d_j} \ee^{V(v)-V(u)}$ (see (A1) in \cite{golosov}) and Chebychev inequality, such that we obtain
\begin{equation}
\label{eq:lem7:1}
P_{\omega,\vert c_j}^{\overline{c}_j}\left(\tau(d_j) > {t /
\log t} \right) \le {\log t \over t} \sum_{c_j \le u \le v \le d_j} \ee^{V(v)-V(u)} \le  C'' {(\log t)^2 \over t} \, \ee^{\gamma \log t},
\end{equation}
on $A_4(t)  \cap F_\gamma(t) \cap  \{ j \le K_t\}.$ For the second term, by (\ref{zeitouni}) we obtain that  the probability $P_{\omega}^{\overline{c}_j}\left(\tau(c_j)<\tau(d_j)\right)$ is less than
\begin{equation}
\label{eq:lem7:2}
{ \sum_{k=\overline{c}_j}^{d_j-1} \ee^{V(k)}} \left(\, \sum_{k=c_j}^{d_j-1} \ee^{V(k)} \right)^{-1}
\le (d_j-c_j) \, \ee^{V(\overline{c}_j)+\gamma \log t - V(c_j)} \le C'' (\log t) \, \ee^{\gamma \log t -{h_t \over 3}},
\end{equation}
on $A_4(t)  \cap F_\gamma(t) \cap  \{ j \le K_t\}.$  Then, assembling (\ref{eq:lem7:1}) and (\ref{eq:lem7:2}) yields
\begin{equation*}
\p\left(\tau(\overline{c}_j,d_j) > {t\over
\log t} \, ; \,  A_4(t)  \, ; \, F_\gamma(t) \right) \le C (\log t) \, \ee^{\gamma \log t -{h_t \over 3}},
\end{equation*}
which concludes the proof of Lemma \ref{l:LT} by choosing $\gamma<1/3.$
\end{proof}

\noindent Now, we need to be sure that the bottom of the deep
valleys are sharp. For $\eta >0$, we introduce the following subsets
of the deep valleys
$$
O_i:=[a_i+1, \overline{c}_i-1]\setminus (b_i-\eta\log t+1, b_i+\eta \log t-1), \qquad i \in \n,
$$
and the event
$$
A_5(t,\eta):=\bigcap_{i=1}^{K_t}\left\{\min_{k\in O_i \cap \z}(V(k)-V(b_i))\ge
C'''\eta\log t\right\},
$$
for a constant $C'''$ (small enough and independent of $\eta$) to be defined later. Then, we have the
following result.
\begin{lemma}\label{l:sharpness}
For all $\eta>0$,
$$
\lim_{t\to \infty} P(A_5(t,\eta))=1.
$$
\end{lemma}
\begin{proof} Observe first that if $\eta>C'',$ then the sets $(O_i, 1 \le i \le K_t)$ are empty on $A_4(t).$ Therefore, Lemma \ref{l:sharpness} is a consequence of Lemma \ref{l:preliminaires}.

Now, let us assume $\eta \le C''.$ The definition of $\overline{c}_i$ implies that $\min_{c_i \le k < \overline{c}_i}(V(k)-V(b_i)) \ge \frac{2}{3} h_t.$ Then, choosing $C'''$ such that $C''' C''<2/3$ implies that $C''' \eta \log t < \frac{2}{3} h_t$ for all large $t,$ which yields
\begin{eqnarray}
\label{eq:lem8:1}
P \left(\bigcap_{i=1}^{K_t}\left\{\min_{c_i \le k < \overline{c}_i}(V(k)-V(b_i))\ge
C'''\eta\log t\right\} \right)=1,
\end{eqnarray}
 for all large $t.$ Then, let us introduce the sets
 $$O'_i:=O_i \cap [b_i,c_i], \qquad O''_i:=O_i \cap [a_i,b_i], \qquad i \in \z,
 $$
and the events
\begin{eqnarray*}
A'_5(t,\eta)&:=&\bigcap_{i=1}^{K_t}\left\{\min_{k\in O'_i \cap \z}(V(k)-V(b_i))\ge
C'''\eta\log t\right\},
\\
A''_5(t,\eta)&:=&\bigcap_{i=1}^{K_t}\left\{\min_{k\in O''_i \cap \z}(V(k)-V(b_i))\ge
C'''\eta\log t\right\}.
\end{eqnarray*}
Now, recalling (\ref{eq:lem8:1}), the proof of Lemma \ref{l:sharpness} boils down to showing that
\begin{eqnarray}
\lim_{t\to \infty} P(A'_5(t,\eta))&=&1,
\label{eq:lem8:2}
\\
\lim_{t\to \infty} P(A''_5(t,\eta))&=&1.
\label{eq:lem8:3}
\end{eqnarray}

Let us first prove (\ref{eq:lem8:2}). Recalling Lemma \ref{l:preliminaires} and Lemma \ref{l:potfluctu}, we only need to prove that it is possible to choose $C'''$ small enough such that for some $\gamma>0$
\begin{eqnarray}
\label{eq:lem8:4}
P\left(\min_{k\in O'_1 \cap \z}(V(k)-V(b_1)) <
C'''\eta\log t  \, ; \, F_\gamma(t) \right)=o((\log t)^{-{1+\kappa\over 2}}),
\end{eqnarray}
when $t \to \infty.$ Now recalling assumption $(a)$ of Theorem \ref{t:main} and denoting by $\mu$  the law of $\log \rho_0,$ we can define the law
$\tilde\mu=\rho_0^\kappa \mu,$ and the law $\tilde
P=\tilde\mu^{\otimes \z}$ which is the law of a sequence of i.i.d.
random variables with law $\tilde\mu$. The definition of $\kappa$
implies that $ \int \log \rho \, \tilde \mu(d\rho) >0.$
Now, let us simplify the notation by writing
$$H:=H_0$$
(where $H_0$ is the height of the first excursion defined by $H_0 := \max_{0 \le k \le e_{1}} V(k)$) and define the hitting time of level $h$ for the potential by
$$
T_h:= \min \{ x \ge 0 \, :\, V(x) \ge h \}, \qquad h>0.
$$
Then, introducing $\tilde F_\gamma(t):= \left\{ -V^{\downarrow}(0,T_H) \le \gamma \log t  \right\},$ we can write that the probability term in (\ref{eq:lem8:4}) is smaller than
\begin{eqnarray}
&&P\left(\min_{\lfloor \eta \log t \rfloor \le k \le T_H}V(k) <
C'''\eta \log t \, ; \, \tilde F_\gamma(t) \, | \, H \ge h_t \right)
\nonumber
\\
&\le& C \ee^{\kappa h_t} P\left(\min_{\lfloor \eta \log t \rfloor \le k \le T_H}V(k) <
C'''\eta \log t \, ; \, \tilde F_\gamma(t)  \, ; \, H \ge h_t \right)
\nonumber
\\
&=& C \tilde E\left[\ee^{-\kappa(V(T_H)- h_t)} {\bf 1}_{\{\min_{\lfloor \eta \log t \rfloor \le k \le T_H} V(k) <
C'''\eta \log t  \, ; \, \tilde F_\gamma(t)  \, ; \, H \ge h_t\} } \right]
\nonumber
\\
&\le&C \tilde P\left(\min_{\lfloor \eta \log t \rfloor \le k \le T_H} V(k) <
C'''\eta \log t  \, ; \, \tilde F_\gamma(t)  \, ; \, H \ge h_t\right),
\label{eq:lem8:5}
\end{eqnarray}
the first inequality being a consequence of (\ref{iglehartthm}) and the equality deduced from Girsanov property. Now, let us introduce $\alpha=\alpha(\eta):=c \eta$ with $c$ satisfying $0<c<\min\{\tilde E \left[V(1)\right]; 1/C''\}$ and $\gamma=\gamma(\eta):= c \eta/2.$ Observe that $\alpha \log t< h_t$ for all large $t,$ so that $T_{\alpha \log t}\le T_{h_t} \le T_H <\infty$ on $\{H \ge h_t\}.$ Now since $c < \tilde E \left[V(1)\right],$ we use Chebychev's inequality in the same manner as is done in the proof of the upper bound in
 Cramer's theorem, see \cite{denhollander}, and obtain that $
\tilde P\left(V(\lfloor \eta \log t \rfloor) < \alpha \log t \right) \le C \exp\{-\eta \tilde I (c) \log t \}= o((\log t)^{-{1+\kappa\over 2}}),$ where $\tilde I (\cdot)$ denotes the convex rate function associated with $V$ under $\tilde P.$  This yields $\tilde P\left( T_{\alpha \log t} \le \lfloor \eta \log t \rfloor  \right)=1-o((\log t)^{-{1+\kappa\over 2}}),$ when $t$ tends to infinity. Therefore, we get
\begin{eqnarray}
&&\tilde P\left(\min_{\lfloor \eta \log t \rfloor \le k \le T_H} V(k) <
C'''\eta \log t \, ; \, \tilde F_\gamma(t)  \, ; \, H \ge h_t\right)\nonumber
\\
&\le& \tilde P\left(\min_{T_{\alpha \log t} \le k \le T_H} V(k) <
C'''\eta \log t \, ; \, \tilde F_\gamma(t)  \, ; \, H \ge h_t\right) + o((\log t)^{-{1+\kappa\over 2}}).
\label{eq:lem8:6}
\end{eqnarray}
Furthermore, observe that on $\tilde F_\gamma(t),$  we have $\min_{T_{\alpha \log t} \le k \le T_H} V(k)\ge (\alpha-\gamma) \log t,$ which yields $\min_{T_{\alpha \log t} \le k \le T_H} V(k)\ge C''' \eta \log t,$ if we choose $C'''$ smaller than $c/2.$ Therefore, for $C'''$ small enough (independently of $\eta \le C''$), we get that the probability term in (\ref{eq:lem8:6}) is null for all large $t.$
Now, assembling  (\ref{eq:lem8:5}) and (\ref{eq:lem8:6}) concludes the proof of (\ref{eq:lem8:2}).

The proof of  (\ref{eq:lem8:3}) is similar but easier. Indeed, we do not have to use Girsanov property to study the potential on $[a_i,b_i].$
\end{proof}

\section{Two versions of a Dynkin type renewal result}
\label{s:dynkin}

We define the sequence of random times $(\tau^*_i)_{i\ge 1}$ as
follows: conditioning on the environment $\w$, $(\tau^*_i)_{i\ge 1}$
is defined as an independent sequence of random variables with the
law of  $\tau(d_i^*)$ under $P^{b_i^*}_{\w, |a_i^*}$, where
$\tau(d_i^*)$ denotes the first hitting time of $d_i^*$ and $
P^{b_i^*}_{\w, |a_i^*}$ is the law of the Markov chain in
environment $\w$, starting from $b_i^*$ and reflected at $a_i^*$.
Hence, under the annealed law $\p$, $(\tau_i^*)_{i\ge 1}$ is an
i.i.d. sequence since the $*$-valleys are independent and
identically distributed. The first step in our proof is to derive
the following result.
\begin{proposition}\label{p:dynkin}
Let $\l_t^*$ be the random integer defined by
$$
\l_t^*:=\sup\{n \ge 0 : \; \tau_1^*+\cdots +\tau_n^*\le t\}.
$$
For all $0\le x_1<x_2 \le 1,$ we have
$$
\lim_{t\to\infty} \p(t(1-x_2)\le \tau_1^*+\cdots +\tau_{\l_t^*}^*\le
t(1-x_1)) ={\sin(\kappa\pi)\over \pi} \int_{x_1}^{x_2}
{ (1-x)^{\kappa -1} x^{-\kappa}} \d x.
$$
For all $0\le x_1<x_2$, we have
$$
\lim_{t\to\infty} \p(t(1+x_1)\le \tau_1^*+\cdots
+\tau_{\l_t^*+1}^*\le t(1+x_2)) ={\sin(\kappa\pi)\over \pi}
\int_{x_1}^{x_2} {\d x\over x^{\kappa} (1+x)}.
$$
\end{proposition}
Observe that the result would exactly be Dynkin's theorem (cf e.g.
Feller, vol II, \cite{feller}, p. 472) if the sequence
$(\tau_i^*)_{i \ge 1}$ was an independent sequence of random variables in the
domain of attraction of a stable law with index $\kappa.$ Here, the
sequence $(\tau_i^*)_{i \ge 1}$ implicitly depends on the time $t$,
since the $*$-valleys are defined from the critical height $h_t$.
We will use the main intermediate result of
\cite{enriquez-sabot-zindy-1} which gives an estimate of the
Laplace transform of $\tau^*_1$ at 0. We deduce from Corollary 2
and Remark 7 of \cite{enriquez-sabot-zindy-1} the following lemma.
\begin{lemma}
\label{l:laplacetau} We have
$$ \E\left[1-\ee^{-\lambda{\tau_1^*\over t}}\right]\sim
2^\kappa{\pi\kappa\over \sin(\pi\kappa)}{C_U\over t^\kappa P(H\ge
h_t)}\lambda^\kappa, \qquad t\to\infty,
$$
for all $\lambda>0$.
\end{lemma}
\begin{proof}
We apply Corollary 2 of \cite{enriquez-sabot-zindy-1} to $n=\lfloor t^\kappa\rfloor$ and $h_n=h_t=
\log t- \log\log t$ which satisfies the condition of Remark 7 of \cite{enriquez-sabot-zindy-1}. The
constant $C_U$ was made explicit in \cite{enriquez-sabot-zindy-2}
but we will not need this value here.

For the convenience of the reader, we give a brief idea of the
arguments of the proof of this formula. Let us simply write
$(a,b,c,d)$ for $(a_1^*, b_1^*, c_1^*, d_1^*)$. The time it takes
to cross the valley can be decomposed in a geometric number of
unsuccessful attempts and a successful attempt, hence we can write
$$
\tau_1^*=\tau(b,d)=F_1+\cdots +F_N+S,
$$
where $N$ is a geometric random variable with parameter
$$
1-p(\w):=P_{\w}^b(\tau(d)<\tau^+(b))=\w_b {\ee^{V(b)}\over
\sum_{x=b}^{d-1} \ee^{V(x)}},
$$
where $\tau^+(b):=\inf\{n>0 : \, X_n =b\}$. The random variables
$(F_i)_{i\ge1}$ are i.i.d. and distributed as $\tau^+(b)$ under
$P_{\w}^b(\, \cdot \, |\tau^+(b)<\tau(d))$ and $S$ is distributed as
$\tau(d)$ under $P_{\w}^b(\, \cdot \,|\tau(d)<\tau^+(b))$. The first
step is to prove that the successful attempt $S$ can be neglected
(this is done in \cite{enriquez-sabot-zindy-1} using some estimates on h-processes). Thus, we can write
$$
\E\big[\ee^{-\lambda{\tau_1^*\over t}}\big]\sim E\left[{1-p(\w)\over 1-p(\w)
E_{\w}\big[\ee^{-{\lambda\over t} F_1}\big]}\right], \qquad t\to\infty.
$$
The second step is to linearize $E_{\w}[\ee^{-{\lambda\over t}
F_1}]$, i.e. to show that it can be replaced by $(1-{\lambda\over
t} E_{\w}[F_1])$ ( using again
estimates on h-processes). This leads to
$$
\E\big[\ee^{-\lambda{\tau_1^*\over t}}\big]\sim E\left[{1\over
1+{\lambda\over t} {p(\w)\over 1-p(\w)} E_{\w}[F_1]}\right], \qquad t\to\infty.
$$
Then we prove that ${p(\w)\over 1-p(\w)} E_{\w}[F_1]$ is of order
$Z=2 \ee^H M_1 M_2$, where $M_1$ and $M_2$ are defined by $M_1:=\sum_{k=a}^c \ee^{-(V(k)-V(b))}$ and
$M_2:=\sum_{k=b}^d \ee^{V(k)-V(c)}.$ Then, we use the main result
of \cite{enriquez-sabot-zindy-2}, where the tail estimate of $Z$
is obtained (see Theorem 2.2).
\end{proof}
\medskip
{\it Proof of Proposition \ref{p:dynkin}.}
The arguments are essentially the same as in \cite{feller}. Let us
introduce $S_0^*=0$ and $S_n^*:=\sum_{i=1}^{n} \tau_i^*,$ for $n \ge
1.$ Then, the inequality $t(1-x_2)\le \tau_1^*+\cdots
+\tau_{\l_t^*}\le t(1-x_1)$ occurs iff $S_n^*=t y$ and  $\tau_{n+1}^*
> t(1-y)$ for some combination $n,$ $y$ such that $1-x_2<y<1-x_1.$
Summing over all $n$ and possible $y$ we get
\begin{eqnarray}
\label{agedevie}
 \p(t(1-x_2)\le S_{\l_t^*}^* \le t(1-x_1))=
\int_{1-x_2}^{1-x_1} {G_t(1-y) \over P(H \ge h_t)}U_t\{ \mathrm{d}
y\},
\end{eqnarray}
where $G_t(x):= P(H \ge h_t) \p(t^{-1}\tau_{1}^* \ge x),$ and
$U_t\{ \mathrm{d} x\}$ denotes the measure associated with $U_t(x):=\sum_{n
\ge 0} \p(t^{-1}S_{n}^* \le x).$ We introduce the measure
$\mathrm{d}H_t(u)$ such that $\int_{x}^{\infty}
\mathrm{d}H_t(u)=G_t(x),$ for all $x \ge 0.$

\begin{lemma}\label{l:convponctG}
For any $x > 0,$ we have
\begin{eqnarray}
\label{eq:convG}
\lim_{t \to \infty} x^\kappa t^\kappa \, G_t(x)=2^\kappa \Gamma(1+\kappa) C_U.
\end{eqnarray}
Moreover, the convergence is uniform on any compact set.
\end{lemma}
\begin{proof} In a first step, observe that $\E[1-\ee^{-\lambda {\tau_1^* \over
t}}]=P(H \ge h_t)^{-1}\int_{0}^{\infty}(1-\ee^{-\lambda u}) \d
H_t(u).$ Recalling Lemma \ref{l:laplacetau}, we obtain
$$
\lim_{t \to \infty} t^\kappa  \int_{0}^{\infty}(1-\ee^{-\lambda u}) \d
H_t(u)=2^\kappa \Gamma(1+\kappa) C_U \Gamma(1-\kappa)
\lambda^\kappa.
$$
Since $\Gamma(1-\kappa) \lambda^{\kappa}=\lambda
\int_{0}^{\infty}\ee^{-\lambda u} u^{-\kappa} \d
 u,$ this implies
\begin{eqnarray}
\label{eq:dynkin1} \lim_{t \to \infty} t^\kappa
\int_{0}^{\infty}(1-\ee^{-\lambda u}) \d H_t(u)= 2^\kappa
\Gamma(1+\kappa) C_U \lambda \int_{0}^{\infty}\ee^{-\lambda u}
u^{-\kappa} \d
 u.
\end{eqnarray}
On the other hand, integrating by parts, we get, for any $t
\ge 0,$
\begin{eqnarray}
\label{eq:dynkin2} \int_{0}^{\infty}(1-\ee^{-\lambda u}) \d H_t(u)=
\lambda \int_{0}^{\infty}\ee^{-\lambda u} G_t(u) \d u.
\end{eqnarray}
Combining (\ref{eq:dynkin1}) and (\ref{eq:dynkin2}) implies that the
measure $t^\kappa G_t(u) \d u$ tends to the measure with density $2^\kappa
\Gamma(1+\kappa) C_U u^{-\kappa}.$ Therefore, we have for all $x \ge
0,$
\begin{eqnarray}
\label{eq:dynkin3} \lim_{t \to \infty} t^\kappa \int_{0}^{x} G_t(u) \d u =
2^\kappa \Gamma(1+\kappa) C_U {x^{1-\kappa} \over 1-\kappa},
\end{eqnarray}
which yields
\begin{eqnarray}
\label{eq:dynkin4} \lim_{\varepsilon \to 0} \ \lim_{t \to \infty} {
\int_{x}^{(1+ \varepsilon) x} G_t(u) \d u  \over \varepsilon
\int_{0}^{x} G_t(u) \d u } = 1-\kappa .
\end{eqnarray}
Moreover, observe that the monotonicity of $G_t(\cdot)$ implies
\begin{eqnarray}
\label{eq:dynkin5} { x G_t((1+ \varepsilon) x)   \over \int_{0}^{x}
G_t(u) \d u }  \le { \int_{x}^{(1+ \varepsilon) x} G_t(u) \d u \over
\varepsilon \int_{0}^{x} G_t(u) \d u } \le {
 x G_t(x)   \over\int_{0}^{x} G_t(u) \d u }.
\end{eqnarray}
Now, combining (\ref{eq:dynkin4}) and (\ref{eq:dynkin5}), we obtain
$$
\liminf_{t \to \infty} {
 x G_t(x)   \over\int_{0}^{x} G_t(u) \d u} \ge 1-\kappa.
$$
Recalling (\ref{eq:dynkin3}), this yields
\begin{eqnarray}
\label{eq:dynkin6} \liminf_{t \to \infty}
 x^{\kappa} t^\kappa G_t(x) \ge 2^\kappa \Gamma(1+\kappa) C_U.
\end{eqnarray}
Similarly, we obtain, for any $\varepsilon>0,$
\begin{eqnarray}
\label{eq:dynkin7} \limsup_{t \to \infty}
 x^{\kappa} t^\kappa G_t((1+\varepsilon)x) \le 2^\kappa \Gamma(1+\kappa) C_U.
\end{eqnarray}
Assembling (\ref{eq:dynkin6}) and (\ref{eq:dynkin7}) and letting $\varepsilon \to 0$ conclude the
proof of (\ref{eq:convG}).

Furthermore, observe that the uniform convergence on any compact set
is a consequence of the monotonicity of $x \mapsto G_t(x),$ the continuity
of the limit and Dini's theorem.
\end{proof}

\begin{lemma}\label{l:convU}
The measure ${P(H \ge h_t)^{-1} \over t^\kappa}U_t\{ \mathrm{d} x\}$ converges
vaguely to the measure  $${1\over \Gamma(\kappa) \Gamma(1+\kappa)
\Gamma(1-\kappa) 2^\kappa C_U} x^{\kappa-1} \mathrm{d} x.$$
\end{lemma}
\begin{proof} Observe first that the Laplace transform $\widehat{U}_t(\lambda):=\int_{0}^{\infty}
 \ee^{-\lambda u} U_t\{ \mathrm{d} u\}$ satisfies $\widehat{U}_t(\lambda)=\sum_{n \ge 0} \e[\ee^{-\lambda {S_n^* \over t}}]=
 (1-\e[\ee^{-\lambda {\tau_1^* \over t}}])^{-1}.$ Therefore, Lemma \ref{l:laplacetau}
 yields
 $$
\lim_{t \to \infty} { P(H \ge h_t)^{-1} \over t^\kappa} \, \widehat{U}_t(\lambda)=
{\lambda^{-\kappa}\over \Gamma(1+\kappa) \Gamma(1-\kappa) 2^\kappa
C_U}.
 $$
 Furthermore, since $\Gamma(\kappa) \lambda^{-\kappa}=\int_{0}^{\infty}\ee^{-\lambda u} u^{\kappa-1} \d
 u,$ we deduce the vague convergence of the measure from the pointwise convergence of the Laplace transforms.
\end{proof}

Now, recalling (\ref{agedevie}), we observe that Lemma
\ref{l:convponctG} together with Lemma \ref{l:convU} imply
\begin{eqnarray*}
\lim_{t \to \infty} \p(t(1-x_2)\le S_{\l_t^*}^* \le t(1-x_1))&=& {1
\over \Gamma(\kappa)\Gamma(1-\kappa)} \int_{1-x_2}^{1-x_1}
(1-y)^{-\kappa} y^{\kappa-1} \d y,
\\
&=& {\sin(\kappa\pi)\over \pi} \int_{x_1}^{x_2}  {(1-y)^{\kappa -1} y^{-\kappa}} \d y.
\end{eqnarray*}
This concludes the proof of the first part of Proposition
\ref{p:dynkin}. The second part of Proposition \ref{p:dynkin} is
obtained using similar arguments.

Recall Lemma \ref{l:IA} which tells that the inter-arrival times are negligible. Now, we will prove that the results of Proposition \ref{p:dynkin} are still true if we consider, in addition, these inter-arrival times.
Let $\delta_1:=\tau(b_1),$ $\tau_1:=\tau(b_1, d_1)$ and
$$
\delta_k:=\tau(d_{k-1}, b_{k}), \qquad     \tau_k:=\tau(b_k, d_k), \qquad k \ge 2.
$$
Moreover, we set
$$
T_k:=\delta_1+\tau_1+\cdots+ \tau_{k-1} +\delta_{k},  \qquad
k\ge 1,
$$
the entering time in the $k$-th deep valley.

\begin{proposition}\label{p:dynkinbis}
Recall $\l_t=\sup \{n \ge 0: \; \tau(b_n) \le t\}.$ Then, we have
 \begin{eqnarray*}
    \p(T_{\l_t}\le t < T_{\l_t}+\tau_{\l_t}) \to 1, \qquad t \to \infty.
 \end{eqnarray*}
For all $0\le x_1< x_2 \le 1$, we have
 \begin{eqnarray*}
\lim_{t\to\infty} \p(t(1-x_2)\le T_{\l_t} \le t(1-x_1))
={\sin(\kappa\pi)\over \pi} \int_{x_1}^{x_2}  {(1-x)^{\kappa -1} x^{-\kappa}} \d x.
 \end{eqnarray*}
For all $0\le x_1<x_2$, we have
 \begin{eqnarray*}
\lim_{t\to\infty} \p(t(1+x_1)\le T_{\l_t+1}\le t(1+x_2)) =
{\sin(\kappa\pi)\over \pi} \int_{x_1}^{x_2}  {\d x\over x^{\kappa}
(1+x)}.
 \end{eqnarray*}
\end{proposition}
\begin{proof} On the event $A(t)\cap DT^* (t)$, we know that the random
times $(\tau_i)_{1\le i\le K^*_t}$ have the same law as the random
times $(\tau^*_i)_{1\le i\le K^*_t}$ defined in Section
\ref{s:dynkin}. If we define $\tilde \l_t :=\sup\{n\ge 0: \;
\tau_1+\cdots + \tau_n\le t\}$, then, using Proposition
\ref{p:dynkin} and Lemma \ref{l:DT}, we get that the result of
Proposition \ref{p:dynkin} is true with $\tau$ and $\tilde \l_t$ in
place of $\tau^*$ and $\l^*_t$. Now, using Lemma \ref{l:IA} we see
that
\begin{eqnarray*}
&&\liminf_{t\to \infty} \p ( \tilde \l_t=\l_t -1\, ;\, T_{\l_t}\le t <
T_{\l_t}+\tau_{\l_t} )
\\
&\ge& \liminf_{t\to\infty} \p( IA(t)\, ;\, \vert t- (\tau_1+\cdots
+\tau_{\tilde \l_t})\vert \ge \xi t ),
\end{eqnarray*}
for all $\xi>0.$ Thus, using Proposition \ref{p:dynkin} (for
$\tilde \l_t$ and $\tau_{i}$) and letting $\xi$ tends to $0,$ we
get that
$$
\lim_{t\to \infty} \p (\tilde \l_t=\l_t -1\, ;\, T_{\l_t}\le t <
T_{\l_t}+\tau_{\l_t})=1.
$$
We conclude the proof by the same type of arguments.
\end{proof}

\section{Proof of part $(i)$ of Theorem \ref{t:main+}: a localization result}
\label{s:localresult}
We follow the strategy developed by Sinai for the recurrent case.
For each valley we denote by $\pi_{i}$ the invariant measure of the
random walk on $[a_i, \overline c_i]$ in environment $\w$, reflected
at $a_i$ and $\overline c_i$ and normalized so that $\pi_i(b_i)=1$.
Clearly, $\pi_i$ is the reversible measure given, for $k\in [b_i+1,
\overline c_i -1],$ by
\begin{eqnarray*}
\pi_i(k)&=&{\w_{b_i}\over 1-\w_{b_i+1}}\cdots {\w_{k-1}\over 1-\w_k}
\\
&=& \w_{b_i} \rho^{-1}_{b_i+1}\cdots \rho^{-1}_{k-1}(\rho_k^{-1}+1)
\\
&\le& \ee^{-(V(k)-V(b_i))}+ \ee^{-(V(k-1)-V(b_i))}.
\end{eqnarray*}
Similarly, $\pi_i(k)\le \ee^{-(V(k)-V(b_i))}+\ee^{-(V(k+1)-V(b_i))}$ for
$k\in [a_i+1, b_i-1]$. Since the walk is reflected at $a_i$ and
$\overline c_i,$ we have $\pi_i(a_i)= \ee^{-(V(a_i+1)-V(b_i))}$ and
$\pi_i(\overline c_i)= \ee^{-(V(\overline c_i -1)-V(b_i))}$. Hence on
the event $A_5(t,\eta)$ we have
$$ \sup \{\pi_i(k) \, ; \,  k \in [a_i, \overline c_i]\setminus (b_i-\eta \log t, b_i+\eta
\log t) \} \le C \ee^{-C'''\eta \log t}=C t^{-C'''\eta}.
$$
Moreover, since $\pi_i$ is an invariant measure and since
$\pi_i(b_i)=1$, we have, for all $k\ge 0,$
$$
P_{\w, |a_i, \overline c_i| }^{b_i}(X_k=x)\le \pi_i(x).
$$
Hence, on the event $A(t)\cap A_5(t,\eta)$  we have, for all $k\ge 0,$
\begin{eqnarray}\label{pii}
P_{\w, |a_i, \overline c_i| }^{b_i}(\vert X_k-b_i\vert > \eta \log
t)\le C ( \log t) t^{-C''' \eta}.
\end{eqnarray}

Let $\xi$ be a positive real, $0<\xi<1$. Then, let us write
\begin{eqnarray*}
&& \liminf_{t\to\infty} \p(\vert X_{t}-b_{\l_t}\vert \le \eta \log t)
\\
&\ge & \liminf_{t\to\infty} \p(\vert X_{t}-b_{\l_t}\vert \le \eta
\log t\, ;\, \l_{t}=\l_{t(1+\xi)})
\\
&\ge& \liminf_{t\to\infty} \p (\l_{t}=\l_{t(1+\xi)}) -
\limsup_{t\to\infty} \p (\vert X_{t}-b_{\l_t}\vert > \eta \log t\,
;\, \l_{t}=\l_{t(1+\xi)}).
\end{eqnarray*}

Considering the first term, we get by using Proposition
\ref{p:dynkinbis},
\begin{eqnarray}
\liminf_{t\to\infty} \p (\l_{t}=\l_{t(1+\xi)} ) &=&
\liminf_{t\to\infty} \p (T_{\l_{t}+1}>t(1+\xi)) \nonumber
\\
&=& {\sin(\kappa\pi)\over \pi} \int_{\xi}^\infty
  {\d x\over x^{\kappa}
(1+x)}. \label{left-term}
\end{eqnarray}

In order to estimate the second term, let us introduce the event
$$
TT(t):=A(t)\cap A_5(t,\eta)\cap DT(t)\cap DT^*(t)\cap A^*(t) \cap
IA(t)\cap LT(t) \cap IT(t),
$$
where $IT(t):= \{T_{\l_t}\le t < T_{\l_t}+\tau_{\l_t}\}.$ Observe
that the preliminary results obtained in Section
\ref{s:preliminaries} together with Proposition \ref{p:dynkinbis}
imply that $\p(TT(t))\to 1,$ when $t \to \infty.$ Then, we have
\begin{eqnarray*}
&& \limsup_{t\to\infty} \p (\vert X_{t}-b_{\l_t}\vert > \eta \log
t\, ;\, \l_{t}=\l_{t(1+\xi)})
\\
&\le& \limsup_{t\to\infty} \p (TT(t)\,; \, \vert X_{t}-b_{\l_t}\vert
> \eta \log t\, ;\, \l_{t}=\l_{t(1+\xi)})
 \\
&\le & \limsup_{t\to\infty} \e\Big[\indic_{TT(t)} \sum_{i=1}^{K_t}
\indic_{ \{ \vert X_{t}-b_{i}\vert
> \eta \log t\, ;\, \l_{t}=\l_{t(1+\xi)}=i \}}\Big].
\end{eqnarray*}
But on the event $TT(t)\cap \{\l_{t}=\l_{t(1+\xi)}=i\}$ we know that
for all $k\in [T_i, t]$ the walk $X_k$ is in the interval
$[a_i,\overline c_i-1]$. Indeed, on the event $LT(t)\cap DT(t)\cap
IA(t)$ we know that once the position $\overline c_i$ is reached
then within a time $t/\log t$ the position $b_{i+1}$ is reached
which would contradict the fact that $\l_{t(1+\xi)}=i$. Hence, we
obtain, for all $i\in \N,$
\begin{eqnarray*}
&&\p\left( TT(t)\,;\, i\le K_t\,;\, \vert X_{t}-b_i\vert
> \eta \log t\, ;\, \l_{t}=\l_{t(1+\xi)}=i \right)
\\
& \le & \E\Big[ \indic_{\{i\le K_t\}}\indic_{A(t)\cap A_5(t,\eta)}
\sup_{k\in [0,t]} P^{b_i}_{\w, |a_i, \overline c_i|}\left(\vert
X_{k}-b_i\vert
> \eta \log t\right)\Big]
\\
&\le& C (\log t) t^{-C'''\eta},
\end{eqnarray*}
where we used the estimate (\ref{pii}) on the event $A(t)\cap
A_5(t,\eta)$. Considering now that, on the event $A(t),$ the number $K(t)$
of deep valleys is smaller than $(\log t)^{{\kappa+1\over 2}}$ we
get
\begin{eqnarray*}
 \limsup_{t\to\infty} \p (\vert X_{t}-b_{\l_t}\vert > \eta \log
t\, ;\, \l_{t}=\l_{t(1+\xi)} ) &\le & \limsup_{t\to\infty} C (\log
t)^{{3+\kappa\over 2}}t^{-C'''\eta}
\\
&=& 0.
\end{eqnarray*}
Then, letting $\xi$ tends to $0$ in (\ref{left-term}) concludes the
proof of part {\it(i)} of Theorem \ref{t:main+}.
\qed

\section{Part $(ii)$ of Theorem \ref{t:main+}: the quenched law of the last visited valley}
\label{s:quenchedlaw}




In order to prove the proximity of the  distributions of $\l_{t}$ and $\l_{t,\omega}^{({\bf e})}$, we go through
$\l_{t}^*=\sup\{n \ge 0, \;\; \tau_1^*+\cdots +\tau_n^*\le t\}$ whose advantage is to involve independent random variables whose laws are clearly identified.

\begin{proposition}\label{p:dTV}
Under assumptions $(a)$-$(b)$ of Theorem \ref{t:main}, we have, for all $\delta>0,$
$$
\lim_{t\to \infty} P\Big(d_{TV}(\l_{t}^{*},\l_{t,\omega}^{({\bf e})})>\delta\Big)=0,$$
where $d_{TV}$ denotes the distance in total variation.
\end{proposition}

\begin{proof}
The strategy is to build a coupling between $\l_{t}^{*}$ and $\l_{t,\omega}^{({\bf e})}$
 such that $$\lim_{t\to\infty}P(P_{0,\omega}(\l_{t}^{*}  \neq  \l_{t,\omega}^{({\bf e})})>\delta)= 0.$$
Let us first associate to the exponential variable ${\bf e}_i$ the following geometric random variable
$$N_i:= \Big\lfloor\big(-\ds{1\over \log(p_i (\omega))}\big){\bf e}_i\Big\rfloor,$$
 where $1-p_i(\omega)$ denotes the probability for the random walk  starting at $b_i$ to go to
 $d_i$ before returning to $b_i$, which is equal to $\omega_b{\ee^{V(b_i)}\over \sum_{x=b_i}^{d_i-1}\ee^{V(x)}}$.
 The parameter of this geometric law is now clearly equal to $1-p_i(\omega)$.

Now one can introduce like in \cite{enriquez-sabot-zindy-1} two random variables $F^{(i)}$ (resp. $S^{(i)}$) whose law are given by the time it takes for the random walk reflected at $a_i$, starting at $b_i,$ to return to $b_i$ (resp. to hit $d_i$) conditional on the event that $d_i$ (resp. $b_i$) is not reached in between.

We introduce now a sequence of   independent copies of $F^{(i)}$
we denote by $(F^{(i)}_n)_{n\geq0}$. The law of $\tau_i^*$ is
clearly the same as $F^{(i)}_1+\dots+F^{(i)}_{N_i}+S^{(i)}$ which
is going now to be compared with $E_{\omega}[\tau_i^*]{\bf e}_i$.

Let us now estimate, for a given $\xi>0$ (small enough),

$\p\Big(\forall i\leq K_t,$

\hfill{$(1-\xi)(F^{(i)}_1+\dots+F^{(i)}_{N_i}+S^{(i)})\leq E_{\omega}[\tau_i^*]{\bf e}_i
<(1+\xi)(F^{(i)}_1+\dots+F^{(i)}_{N_i}+S^{(i)})\Big)$}

$\geq\p\left(\forall i\leq K_t,\quad(1-{\xi\over2})(F^{(i)}_1+\dots+F^{(i)}_{N_i})\leq E_{\omega}[\tau_i^*]{\bf e}_i
<(1+{\xi\over2})(F^{(i)}_1+\dots+F^{(i)}_{N_i})\right)$
\begin{eqnarray}\label{smaller-bound}
\quad-\,\p\Big(\exists i\leq K_t,\quad S^{(i)}>{\xi \over3}(F^{(i)}_1+\dots+F^{(i)}_{N_i})\Big).
\end{eqnarray}

Let us first treat the second quantity of the rhs of (\ref{smaller-bound}). For this purpose, we need an upper bound for $E_\omega[S^{(i)}]$ which is obtained exactly like in Lemma $13$ of  \cite{enriquez-sabot-zindy-1} and can be estimated by controlling the size of the falls (resp. rises)  of the potential during its rises from $V(b_i)$ to $V(c_i)$
(resp. falls from $V(c_i)$ to $V(d_i)$), see Lemma \ref{l:potfluctu}. Indeed, the random variable $S^{(i)}$ concerns actually the random walk which is {\it conditioned} to hit $d_i$ before $b_i$. Therefore, this involves an $h$-process which can be viewed as a random walk in a modified potential between $b_i$ and $d_i$. This modified potential
has a decreasing trend (which encourages the particle to go to the right), and the main contribution to  $S^{(i)}$ comes from the small risings of this modified potential along its global fall.


More precisely, the particle starting at $b_i$ which is conditioned to hit $d_i$ before returning to $b_i$ moves like a particle in the  modified random potential $\bar{V}^{(i)}$ defined as follows: for all $b_i\le x < y \le d_i,$
\begin{eqnarray}
\label{Vbareq0} \bar{V}^{(i)}(y)-\bar{V}^{(i)}(x)=\left(V(y)-V(x)\right)+ \log
\bigg({ g^{(i)}(x) \, g^{(i)}(x+1) \over g^{(i)}(y) g^{(i)}(y+1)} \bigg),
\end{eqnarray}
where $g^{(i)}(x):=
P_{\omega}^{x}(\tau(d_i) < \tau(b_i)).$
The expectation of $S^{(i)}$ is given by the usual formula (see \cite{zeitouni}), so that
$$E_{\omega}[S^{(i)}] \le1 + \sum_{k=b_i+1}^{d_i}
    \sum_{l=k}^{d_i} \ee^{\bar{V}^{(i)}(l)-\bar{V}^{(i)}(k)}.$$
We are therefore concerned by the largest rise of $\bar{V}^{(i)}$ inside the interval $[b_i, d_i]$.
We first notice that, by standard arguments, for any $b_i\le x < y \le d_i,$
\begin{eqnarray}
\label{Vbareq1} { g^{(i)}(x) \, g^{(i)}(x+1) \over g^{(i)}(y) \, g^{(i)}(y+1)}={
\sum_{j=b_i}^{x-1} \ee^{V(j)}
 \, \sum_{j=b_i}^{x} \ee^{V(j)} \over \sum_{j=b_i}^{y-1} \ee^{V(j)}
 \, \sum_{j=b_i}^{y} \ee^{V(j)}}\le 1.
\end{eqnarray}
Therefore, we obtain for any $b_i\le x < y \le d_i$
\begin{eqnarray}
\label{Vbareq2}
   \bar{V}^{(i)}(y)-\bar{V}^{(i)}(x) \le V(y)-V(x).
\end{eqnarray}
This allows to bound the largest rise of $\bar{V}^{(i)}$ on the interval $[c_i, d_i]$ by the largest rise of $V$ on this interval.

Concerning the largest rise of $\bar{V}^{(i)}$ on the interval $[b_i,c_i]$, we notice, taking into account the small size of the fluctuations of $V$ described in Lemma  \ref{l:potfluctu},   (\ref{Vbareq1}) and (\ref{Vbareq2}), that for all $\eta>0$, for all $\omega\in A_4(t)\cap F_\eta(t)$, and for all $i\leq K_t$, the difference $\bar{V}^{(i)}(y)-\bar{V}^{(i)}(x)$ is less than or equal to
\begin{eqnarray*}
 &&[V(y)-\max_{b_i \le j \le y} V(j)]-[V(x)-\max_{b_i \le j \le
x} V(j)]+O( \log \log t)
\\
&\le& \eta \log t + O( \log \log t).
\end{eqnarray*}
This reasoning  yields for all $\eta>0 $ that, for all $\omega\in A_4(t)\cap F_\eta(t),$
$$\forall i\leq K_t,\quad E_\omega[S^{(i)}]\leq t^\eta.$$
This implies, by the Markov inequality, that, for all $\eta>0$ and all $\omega\in A_4(t)\cap F_\eta(t),$
$$\forall i\leq K_t,\quad P_\omega(S^{(i)}>t^{2\eta})<{1\over t^{\eta}}.$$
On the other hand, we have
$$P_\omega(F^{(i)}_1+\dots+F^{(i)}_{N_i}<t^{2\eta})\leq P_\omega(N_i<t^{2\eta})=1-p_i(\omega)^{\lfloor t^{2\eta} \rfloor}=O\left({t^{2\eta}\log t\over t}\right),
$$
the last equality coming from the definition of  $h_t:= \log t-\log(\log t)$, which implies that $1-p_i(\omega)$ is smaller than ${\log t\over  t}.$
Hence, since  $A_2(t) = \{K_t\le (\log t)^{{1+\kappa\over 2}}\}$ satisfies $P(A_2(t))\to1$ (see Lemma \ref{l:preliminaires}), we obtain
$$\lim_{t\to+\infty}P\Big(\forall i\leq K_t,\quad P_\omega(F^{(i)}_1+\dots+F^{(i)}_{N_i}< t^{2\eta})\leq{1\over t^{{1\over2}-2\eta}}\Big)=1.$$
Gathering these two informations on $S^{(i)}$ and $F^{(i)}_1+\dots+F^{(i)}_{N_i}$, we obtain
$$\lim_{t\to+\infty}\p\Big(\forall i\leq K_t,\quad S^{(i)}<{\xi \over3}(F^{(i)}_1+\dots+F^{(i)}_{N_i})\Big)=1,$$
for all $\xi>0,$ which treats the second quantity of the rhs  of (\ref{smaller-bound}).

The first quantity of the rhs of (\ref{smaller-bound})  is treated by going through
$$\p\Big((1-{\xi\over4})N_iE_{\omega}[F^{(i)}]\leq F^{(i)}_1+\dots+F^{(i)}_{N_i}\leq (1+{\xi\over4})N_iE_{\omega}[F^{(i)}]\Big),$$ which, for all $\eta>0$,  is larger than
$$1-\p\left(\left\{\left\vert { F^{(i)}_1+\dots+F^{(i)}_{N_i}\over N_i}-E_\omega[F^{(i)}]
\right\vert> {\xi\over4} E_\omega[F^{(i)}] \right\}\cap \{N_i\neq 0 \} \cap \{E_\omega[(F^{(i)})^2] \leq t^\eta\}\right) $$
$$ -P(E_\omega[(F^{(i)})^2] \geq t^\eta),$$
which is in turn, using the Bienaim\'e-Chebychev's inequality, larger than
$$1- E\Big[E({t^\eta\over N_i}{\bf 1}_{\{N_i\neq0\}}\, |\, N_i){16\over \xi^2 E_\omega \left[F\right] ^2}\Big]
-P(E_\omega[(F^{(i)})^2] \geq t^\eta)$$
$$ \geq 1-{16 t^\eta\over\xi^2}\E\Big[{1\over N_i}{\bf 1}_{\{N_i\neq0\}}\Big]-P(E_\omega[(F^{(i)})^2] \geq t^\eta).$$
Now, we use again the reasoning based on $h$-processes to get an upper bound for $E_\omega[(F^{(i)})^2] $. Like in the success case, the particle starting at $b_i$ which is conditioned to hit $b_i$ before returning to $d_i$ moves like a particle in the  modified random potential $\widehat{V}^{(i)}$ defined as follows: for all $a_i\le x < y \le d_i,$
\begin{eqnarray}
\label{Vbareq0} \widehat{V}^{(i)}(y)-\widehat{V}^{(i)}(x)=\left(V(y)-V(x)\right)+ \log
\bigg({ h^{(i)}(x) \, h^{(i)}(x+1) \over h^{(i)}(y) h^{(i)}(y+1)} \bigg),
\end{eqnarray}
where $h^{(i)}(x):=
P_{\omega}^{x}(\tau(b_i) < \tau(d_i))$ (notice that $V$ and $\widehat V^{(i)}$ coincide on the interval $[a_i,b_i]$).

It happens now that $E_\omega[(F^{(i)})^2] $ can be computed explicitly in terms of $\widehat V^{(i)}$ (see Lemma 12 in \cite{enriquez-sabot-zindy-1}), and is bounded by a constant times $(d^{(i)}-a^{(i)})^2$ times the exponential of the maximum of the largest rise of $V$ on $[a_i, b_i]$ and the largest fall of $\widehat V^{(i)}$ on $[b_i, d_i]$, which are treated in a similar way as the fluctuations of $\bar V^{(i)}$, above. So, we get
 $$\forall \eta>0, \quad P(E_\omega[(F^{(i)})^2] \geq t^\eta)=o\Big({1\over (\log t)^{2}}\Big).$$
Moreover, we have
$$E\Big[{1\over N_i}{\bf 1}_{\{N_i\neq0\}}\Big]= E\Big[-{1-p_i(\omega)\over p_i(\omega)}\log(1-p_i(\omega))\Big]=O \Big({(\log t)^2\over t}\Big).$$
As a result, we obtain $$\p\Big((1-{\xi\over4})N_iE_{\omega}[F^{(i)}]\leq F^{(i)}_1+\dots+F^{(i)}_{N_i}\leq (1+{\xi\over4})N_iE_{\omega}[F^{(i)}]\Big)= 1-o\Big({1\over (\log t)^2}\Big).$$

Now, the second step in the estimation of the first quantity of the rhs of  (\ref{smaller-bound}) is the examination, for $\xi>0$, of
$$\p\Big((1-{\xi\over4})N_iE_{\omega}[F^{(i)}]\leq E_\omega[\tau_i]{\bf e}_i \leq (1+{\xi\over4})N_iE_{\omega}[F^{(i)}]\Big),$$
i.e.
$$\p\Big((1-{\xi\over4})N_iE_{\omega}[F^{(i)}]\leq (E_\omega[N_i]E_{\omega}[F^{(i)}]+E_\omega[S^{(i)}]){\bf e}_i \leq (1+{\xi\over4})N_iE_{\omega}[F^{(i)}]\Big).$$
Neglecting again, like above, the contribution of $S^{(i)}$ we are back to prove that
$$\p\Big((1-{\xi\over4})\Big\lfloor(-\ds {1\over \log(p_i (\omega))}){\bf e}_i\Big\rfloor\leq {p_i (\omega)\over 1-p_i(\omega)}
{\bf e}_i\leq (1+{\xi\over4})\Big\lfloor(-\ds { 1 \over \log(p_i (\omega))}){\bf e}_i\Big\rfloor\Big)=1-o\Big({1\over (\log t)^2}\Big),$$
which is a direct consequence of $1-p_i(\omega)\leq{\log t\over  t}$ allied with
$$P^{({\bf e})} \Big({\bf e}_i>{\log t\over t}\Big)=1-o\Big({1\over (\log t)^2}\Big).$$
Now, since $P(K_t\le (\log t)^{{1+\kappa\over 2}})\to 1,$ when $t \to \infty,$ this concludes the proof that the rhs (and therefore the lhs) of  (\ref{smaller-bound}) tends to $1$ when $t$ tends to infinity. Indeed, we  obtain
$$\p\Big(\forall i\leq K_t,\, (1-\xi)(F^{(i)}_1+\dots+F^{(i)}_{N_i}+S^{(i)})\leq E_{\omega}[\tau_i^*]{\bf e}_i
<(1+\xi)(F^{(i)}_1+\dots+F^{(i)}_{N_i}+S^{(i)})\Big)\to1,$$
from which we deduce
$$\p\Big(\forall i\leq K_t,\; (1-\xi)(\tau_1^*+\cdots+\tau_i^*)\leq \sum_{k=1}^iE_{\omega}[\tau_k^*]{\bf e}_k
<(1+\xi)(\tau_1^*+\cdots+\tau_i^*)\Big)\to1.
$$
Moreover, we use the fact that
$$
E_{\omega}[\tau_k^*]=W^*_k- (d_k^*-b_k^*),
$$
where (cf for example \cite{zeitouni}, formula (2.1.14))
$$ W_k^*=2 \sum_{{a^*_k\leq m\leq n}\atop{b^*_k\leq n\leq
d^*_k}}e^{V_\omega(n)-V_\omega(m)}.
$$
 Since on the event
$A_4(t)$ we have for all $k\le K_t$, $d_k^*-b_k^*\le C''\log t$,
we see that on $A_4(t)$ we have $(1-C''{(\log t)^2 \over t})
W_k^*\le E_{\omega}[\tau_k^*] \le W_k^*$, since $W_k^*\ge e^{h_t}$ by
definition. Hence, it implies that
$$
\p\Big(\forall i\leq K_t,\; (1-\xi)(\tau_1^*+\cdots+\tau_i^*)\leq
\sum_{k=1}^i W^*_k {\bf e}_k
<(1+\xi)(\tau_1^*+\cdots+\tau_i^*)\Big)\to1.
$$
Applying this, for $i=\l_t^*$ and $i=\l_{t,\omega}^{({\bf e})}$ we get respectively
that, for all $\xi>0$, $$\p\Big(\l_t^*\leq\l_{{t\over 1-\xi},\omega}^{({\bf e})}\Big)\to 1\quad\hbox{and} \quad
\p(\l_{t,\omega}^{({\bf e})}\leq \l_{t(1+\xi)}^*)\to 1.$$
We conclude now the proof by reminding that $\lim_{\xi\to0}\p(\l_t^*=\l_{t(1+\xi)}^{*})=1$
as well as  $\lim_{\xi\to0}\p(\l_{t,\omega}^{(\bf e)}=\l_{t(1+\xi),\omega}^{(\bf e)})=1.$
\end{proof}

\medskip
\medskip
{\it Proof of Part (ii) of Theorem \ref{t:main+}.} The passage from Proposition \ref{p:dTV} to Part $(ii)$ of  Theorem \ref{t:main+}  is of the same kind as
the passage from Proposition \ref{p:dynkin} to Proposition \ref{p:dynkinbis}.\qed

 \section{Proof of Theorem \ref{t:main}}
\label{s:proof}
We fix $h>1$ and $\eta>0$ ($\eta$ was used to define the event
$A_5(t,\eta)$ before Lemma \ref{l:sharpness}). Let us introduce the event
$$
TT(t, h):= TT(t) \cap \{X_t-b_{\l_t} \le {\eta \over 2} \log t\} \cap
\{X_{th}-b_{\l_{th}} \le {\eta \over 2} \log t\},
$$
whose probability tends to 1, when $t$ tends to infinity (it is a
consequence of Section \ref{s:preliminaries} together with part $(i)$ of
Theorem \ref{t:main+}). Then, we easily have
$$
\left(\{ \l_{th}=\l_t\} \cap  TT(t,h) \right)  \subset \left( \{\vert
X_{th}-X_t\vert\le \eta \log t\} \cap TT(t,h)\right).
$$
Moreover, observe that on $TT(t),$ $\l_{th}> \l_t$ implies that $\vert
b_{\l_{th}}-b_{\l_{t}} \vert \ge t^{\kappa/2}$ (by definition of
$A_3(t)$). Therefore, we get
$$\left( \{\vert X_{th}-X_t\vert\le \eta \log t\}
\cap TT(t,h)\right)  \subset \left(\{ \l_{th}=\l_t\} \cap
TT(t,h)\right),
$$
for all large $t.$ Thus, since Proposition \ref{p:dynkinbis} implies
that $\lim_{t\to\infty} \p (\l_{th}=\l_t)$ exists, we obtain
\begin{eqnarray*}
\lim_{t\to\infty} \p( \vert X_{th} -X_t \vert \le \eta \log t) &=
& \lim_{t\to\infty} \p (\l_{th}=\l_t)
\\
&=& \lim_{t\to\infty} \p( T_{\l_t+1}\ge th)
\\
&=&  {\sin(\kappa\pi)\over \pi} \int_{h-1}^{\infty} {\d x \over
x^{\kappa} (1+x)}
\\
&=&  {\sin(\kappa\pi)\over \pi} \int_{0}^{1/h} y^{\kappa-1} (1-y)^{-\kappa}\d y,
\end{eqnarray*}
which concludes the proof of Theorem \ref{t:main}.
\qed

\end{document}